\documentclass{article}

\usepackage{amsfonts, amsmath, amssymb, amsthm, amscd}
\usepackage{ascmac}
\usepackage[dvipdfm]{color}
\usepackage{verbatim}
\usepackage[dvips]{graphicx}
\usepackage{graphics}
\usepackage[all,2cell]{xypic}

\objectmargin+{1mm}
\labelmargin+{0.8mm}
\SelectTips{cm}{12}

\UseAllTwocells

\theoremstyle{plain}  
\newtheorem{thm}{Theorem}[section]

\newtheorem{cor}[thm]{Corollary}
\newtheorem{lem}[thm]{Lemma}
\newtheorem{prop}[thm]{Proposition}

\theoremstyle{definition}

\newtheorem{df}[thm]{Definition}
\newtheorem{ex}[thm]{Examples}

\newtheorem{rem}[thm]{Remark}

\newtheorem*{conv}{Conventions}
\newtheorem{para}[thm]{}

\theoremstyle{remark}

\DeclareMathOperator{\cB}{\mathcal{B}}
\DeclareMathOperator{\cC}{\mathcal{C}}
\DeclareMathOperator{\calD}{\mathcal{D}}
\DeclareMathOperator{\cE}{\mathcal{E}}
\DeclareMathOperator{\cF}{\mathcal{F}}
\DeclareMathOperator{\cG}{\mathcal{G}}

\DeclareMathOperator{\calR}{\mathcal{R}}

\DeclareMathOperator{\cT}{\mathcal{T}}

\DeclareMathOperator{\bB}{\mathbf{B}}
\DeclareMathOperator{\bC}{\mathbf{C}}

\DeclareMathOperator{\bE}{\mathbf{E}}
\DeclareMathOperator{\bF}{\mathbf{F}}
\DeclareMathOperator{\bG}{\mathbf{G}}

\DeclareMathOperator{\bZ}{\mathbf{Z}}

\DeclareMathOperator{\bbK}{\mathbb{K}}

\DeclareMathOperator{\fa}{\mathfrak{a}}

\DeclareMathOperator{\fb}{\mathfrak{b}}

\DeclareMathOperator{\fF}{\mathfrak{F}}

\UseAllTwocells

\def\sn{\smallskip\noindent}

\newcommand{\cf}{\textrm{cf.}\;}
\newcommand{\Ch}{\operatorname{\bf Ch}}

\newcommand{\Coker}{\operatorname{Coker}}

\newcommand{\Hom}{\operatorname{Hom}}

\newcommand{\id}{\operatorname{id}}

\newcommand{\isoto}{\overset{\scriptstyle{\sim}}{\to}}

\newcommand{\linc}{\hookleftarrow}

\newcommand{\onto}[1]{\stackrel{#1}{\to}}
\newcommand{\op}{\operatorname{op}}

\newcommand{\qis}{\operatorname{qis}}

\newcommand{\rdef}{\twoheadrightarrow}
\newcommand{\RelEx}{\operatorname{\bf RelEx}}
\newcommand{\rinc}{\hookrightarrow}
\newcommand{\rinf}{\rightarrowtail}

\title{Delooping of relative exact categories}
\date{}

\author{Toshiro Hiranouchi   \and
        Satoshi Mochizuki 
}

\begin{document}

\maketitle

\begin{abstract}
We introduce a delooping model of relative exact categories. 
It gives us a condition that  
the negative $K$-group of a relative exact category becomes trivial. 

\sn

\sn
Subclass:19D35 \and 18E30
\end{abstract}

\section*{Introduction}
The negative $K$-theory $\bbK(\cE)$ for an exact category $\cE$ 
is introduced in \cite{Sch04}, \cite{Sch06} and \cite{Sch11} 
by M.\ Schlichting, and also 
for a differential graded category and 
for a stable infinity category 
is innovated by C.\ Cisikinsi and G.\ Tabuada, 
A.J.\ Blumberg and D.\ Gepner and G.\ Tabuada 
in \cite{CT11} and \cite{BGT10} respectively. 
This generalizes 
the definition of Bass, Karoubi, Pedersen-Weibel, Thomason, Carter and Yao. 
The first motivation of our work 
is to investigate some vanishing conjectures 
of such negative $K$-groups: \\
$\mathrm{(a)}$ 
For any noetherian scheme $X$ of Krull dimension $d$, 
$K_{-n}(X)$ is trivial for $n>d$ (\cite{Wei80}). \\
$\mathrm{(b)}$ 
The negative $K$-groups of a small abelian category is trivial (\cite{Sch06}). \\
$\mathrm{(c)}$ 
For a finitely presented group $G$, $K_{-n}(\bZ G) = 0$ for $n>1$ 
(\cite{Hsi84}). 

\sn
In \cite{Sch06} Corollary~6, 
it was given 
a description of $\bbK_{-1}(\cE)$ 
and 
a condition on vanishing of $\bbK_{-1}(\cE)$ 
for an (essentially small) exact category $\cE$ 
in terms of its {\it unbounded} derived category $\calD(\cE)$: 
We have 
$\bbK_{-1}(\cE)=\bbK_0(\calD(\cE))$ 
and $\bbK_{-1}(\cE)$ is trivial 
if and only if $\calD(\cE)$ is idempotent complete 
(= Karoubian in the sense of \cite{TT90}, A.6.1). 
To extend this observation, we shall introduce 
the notion of {\it{higher derived categories}} $\calD_n(\cE)$ 
and 
show the following theorem:

\begin{thm}[\bf A special case of Corollary~\ref{cor:obst}]
\label{thm:main}
For an exact category $\cE$, 
we have 
$\bbK_{-n}(\cE)= \bbK_0(\calD_n(\cE))$. 
Moreover, $\bbK_{-n}(\cE)$ is trivial if and only if 
$\calD_n(\cE)$ is idempotent complete. 
\end{thm}

\sn
Recall that the derived category $\calD(\cE)$ 
is the triangulated category obtained by 
formally inverting quasi-isomorphisms in 
the category of chain complexes $\Ch(\cE)$. 
The pair $(\Ch(\cE),\qis)$ of the category of chain complexes on $\cE$ 
and the class of all quasi-isomorphisms in $\Ch(\cE)$ forms 
a complicial Waldhausen exact category. 
More generally, 
for a pair $\bE=(\cE,w)$ of 
an exact category $\cE$ and a class of morphisms $w$ in $\cE$ 
which is closed under finite compositions and 
satisfies the {\it strict axiom} (\cf \ref{para:ext axiom}), 
we define a class of weak equivalences $qw$ 
in the category of chain complexes 
$\Ch(\cE)$,  
which is called {\it{quasi-weak equivalences}} 
associated with $w$. 
If $w$ is the class of isomorphisms in $\cE$, 
then $qw$ is just the class of quasi-isomorphisms on $\Ch(\cE)$. 
The derived category $\calD(\bE)$ of $\bE$ is obtained by 
formally inverting the quasi-weak equivalences 
in $\Ch(\cE)$. 
Put $\Ch(\bE)=(\Ch(\cE),qw)$ and 
one can define the class of weak equivalences 
in $\Ch_n(\bE) := \Ch(\Ch_{n-1}(\bE))$ inductively. 
The $n$-th derived category $\calD_n(\bE)$ associated with $\bE$,  
is the derived category of $\Ch_n(\bE)$. 
We also obtain the following theorems 
on the negative $K$-theory $\bbK(\bE)$ for a 
very strict consistent relative exact category $\bE$ 
(for definition, see \ref{para:ext axiom} and \ref{para:consis axiom}): 

\begin{thm}[\bf A special case of Theorem~\ref{thm:abstract main theorem}]
\label{thm:main thm 2}
$\bbK(\Ch(\bE))\isoto\Sigma\bbK(\bE)$, 
where $\Sigma$ is a suspension functor 
on the stable category of spectra.
\end{thm}

\sn
The organization of this note is as follows: 
In Section\ \ref{sec:derived categories}, 
we define the derived categories of $\bE$ and introduce the notion 
of quasi-weak equivalences. 
We will prove Theorems~\ref{thm:main thm 2} and \ref{thm:main} in 
Sections\ \ref{sec:Localizing theory} and \ref{sec:Main theorems} respectively.

\begin{conv}
Throughout the paper, we use the letters $\cE$ and $w$ 
to denote an essentially small exact category and 
a class of morphisms in $\cE$. 
We write $\bE$ for the pair $(\cE,w)$. 
For any category $\cC$, 
we denote the class of all isomorphisms in $\cC$ 
by $i_{\cC}$ or simply $i$. 
For any additive category $\cB$, 
we denote the category of bounded 
(resp. unbounded, bounded above and bounded below) chain complexes on $\cB$ 
by $\Ch^b(\cB))$ (resp. $\Ch(\cB)$, $\Ch^{+}(\cB)$ and  $\Ch^{-}(\cB)$). 
For any additive category $\bB$ we denote the idempotent completion of 
$\bB$ by $\bB^{\sim}$. 
For any essentially small triangulated category $\cT$, 
we denote the Grothendieck group of $\cT$ (resp. $\cT^{\sim}$) by 
$K_0(\cT)$ (resp. $\bbK_0(\cT)$). 
We say that a sequence of triangulated categories 
$\cT \onto{i} \cT' \onto{j} \cT''$ is {\it exact} if 
$i$ is fully faithful, the composition $ji$ is zero and 
the induced functor from $j$, $\cT'/\cT \to \cT''$ is {\it cofinal}. 
The last condition means that it is fully faithful and 
every object of $\cT''$ is 
a direct summand of an object of $\cT'/\cT$. 
For any complicial exact category with weak equivalence 
$\bC=(\cC,w)$, 
we write $\cT\bC=\cT(\cC,w)$ and $\fF\bC$ for 
the associated triangulated category and the countable envelope of $\bC$ 
(See \cite[\S 3.2]{Sch11}). 
\end{conv}

\section{Derived categories}
\label{sec:derived categories}

\begin{para}[\bf Relative exact categories]
\label{para:ext axiom}
$\mathrm{(1)}$ 
A {\it relative exact category} $\bE=(\cE,w)$ is a pair of 
an exact category $\cE$ 
with a specific zero object $0$ 
and a class of morphisms in $\cE$ 
which satisfies the following two axioms.\\
{\bf (Identity axiom).} 
For any object $x$ in $\cE$, 
the identity morphism $\id_x$ is in $w$.\\
{\bf (Composition closed axiom).} 
For any composable morphisms $\bullet \onto{a} \bullet \onto{b} \bullet$ 
in $\cE$, if $a$ and $b$ are in $w$, then $ba$ is also in $w$.\\
$\mathrm{(2)}$ 
A {\it relative exact functor} between relative exact categories 
$f:\bE=(\cE,w) \to (\cF,v)$ is 
an exact functor $f:\cE \to \cF$ such that 
$f(w)\subset v$ and $f(0)=0$. 
We denote the category of relative exact categories and relative exact functors 
by $\RelEx$.\\
$\mathrm{(3)}$ 
We write $\cE^w$ for the full subcategory of $\cE$ 
consisting of those object $x$ such that the canonical morphism 
$0 \to x$ is in $w$. 
We consider the following axioms.\\
{\bf (Strict axiom).} 
$\cE^w$ is an exact category such that 
the inclusion functor $\cE^w \rinc \cE$ is exact 
and reflects exactness.\\
{\bf (Very strict axiom).} 
$\bE$ satisfies the strict axiom and 
the inclusion functor $\cE^w \rinc \cE$ induces a fully faithful 
functor $\calD^b(\cE^w) \rinc \calD(\cE)$ on the bounded derived categories.\\
We denote the category of strict (resp. very strict) relative exact categories 
by $\RelEx_{\operatorname{strict}}$ 
(resp. $\RelEx_{\operatorname{vs}}$).\\
$\mathrm{(4)}$ 
A {\it relative natural equivalence} $\theta:f \to f'$ 
between relative exact functors $f$, $f':\bE=(\cE,w) \to \bE'=(\cE',w')$ 
is a natural transformation $\theta:f \to f'$ such that $\theta(x)$ is in $w'$ 
for any object $x$ in $\cE$. 
Relative exact functors $f$, $f':\bE \to \bE'$ are {\it weakly homotopic} 
if there is a zig-zag sequence of ralative natural equivalences 
connecting $f$ to $f'$. 
A relative functor $f:\bE \to \bE'$ is a {\it homotopy equivalence} 
if there is a relative exact functor $g:\bE' \to \bE$ 
such that $gf$ and $fg$ are 
weakly homotopic to identity functors respectively.\\
$\mathrm{(5)}$ 
A functor $F$ from a full subcategory $\calR$ of $\RelEx$ to 
a category $\cC$ is {\it categorical homotopy invariant} 
if for any relative exact functors $f$, $f':\bE \to \bE'$ 
such that $f$ and $f'$ are weakly homotopic, 
we have the equality $F(f)=F(f')$. 
\end{para}

\begin{prop}
\label{prop:fully faithful}
Let $f:\cF \to \cG$ be an exact functor between 
exact categories. 
If the induced functor $\calD^{\#}(f):\calD^{\#}(\cF) \to \calD^{\#}(\cG)$ 
is fully faithful for some $\#\in \{b,\pm,\operatorname{nothing}\}$, 
then $\calD^{\#}(f)$ is fully faithful 
for any $\#\in \{b,\pm,\operatorname{nothing}\}$.
\end{prop}

\begin{proof}[\bf Proof]
Since we have the fully faithful embeddings 
$\calD^b(\cF) \to \calD^{\pm}(\cF)$ and 
$\calD^{\pm}(\cF) \to \calD(\cF)$ 
and the equality $\calD^{+}(\cF)={(\calD^{-}\cF^{\op})}^{\op}$, 
we only need to check that 
if $\calD^b(f)$ (resp. $\calD^{-}(f)$) 
is fully faithful, then 
$\calD^{+}(f)$ (resp. $\calD(f)$) is also. 
For any obejects $x$ and $y$ in $\Ch^{-}(\cF)$ (resp. $\Ch(\cF)$), 
there are sequences $\fa=\{a_0\rinf a_1\rinf a_2\rinf \cdots\} $ 
and $\fb=\{b_0\rinf b_1\rinf b_2\rinf \cdots\}$ 
of inflations in $\Ch^b\cF$ (resp. $\Ch^{-}\cF$) 
such that $x\isoto\fa $ and $y\isoto\fb$ in $\cT(\fF\Ch^{-}\cF,\qis)$ 
(resp. $\cT(\fF\Ch\cF,\qis)$). 
Since we have the fully faithful embeddings 
$\calD^{\#}\cF \rinc \cT(\fF\Ch^{\#}\cF,\qis) \linc \cT(\fF\Ch^{\#'}\cF,\qis)$ 
where $\#=-$ and $\#'=b$ (resp. $\#=\operatorname{nothing}$ and $\#'=+$) 
by \cite{Sch06} Proposition~1 and Theorem~3, 
we regard both $x$ and $y$ as an objects in 
$\cT(\fF\Ch^b\cF,\qis)$ (resp. $\cT(\fF\Ch^{+}\cF,\qis)$) 
and we have the natural isomorphism 
$\Hom_{\calD^{\#}\cF}(x,y)\isoto \Hom_{\cT(\fF\Ch^{\#'}\cF,\qis)}(x,y)$ 
where $\#=-$ and $\#'=b$ (resp. $\#=\operatorname{nothing}$ and $\#'=+$). 
On the other hand, 
the induced functor 
$\cT(\fF\Ch^{\#}(f),\qis):\cT(\fF\Ch^{\#}(\cF),\qis) 
\to \cT(\fF\Ch^{\#}(\cG),\qis)$ 
where $\#=b$ (resp. $\#=+$) is fully faithful 
by \cite{Sch06} Corollary~2 and Proposition~1. 
Hence we obtain the result. 
\end{proof}

\begin{para}[\bf Derived category]
\label{para:derivd cat}
We define the {\it derived categories} of 
a strict relative exact category $\bE=(\cE,w)$ by the following formula
$$\calD^{\#}(\bE):=\Coker(\calD^{\#}(\cE^w) \to \calD^{\#}(\cE))$$
where $\# =b$, $\pm$ or nothing. 
Namely $\calD^{\#}(\bE)$ is a Verdier quotient of $\calD^{\#}(\cE)$ 
by the thick subcategory of $\calD^{\#}(\cE)$ 
spanned by the complexes in $\Ch^{\#}(\cE^w)$.
\end{para}

\begin{df}[\bf Exact sequence]
\label{df:exact sequence}
A sequence $\bE \onto{u} \bF \onto{v} \bG$ 
of strict relative exact categories is 
{\it exact}
if the induced sequence of triangulated categories 
$\calD^b(\bE) \onto{\calD^b(u)} \calD^b(\bF) \onto{\calD^b(v)} \calD^b(\bG)$ is 
exact. 
We sometimes denote the sequence above by $(u,v)$. 
For a full subcategory $\calR$ of $\RelEx_{\operatorname{strict}}$, 
we let $E(\calR)$ denote the category of 
exact sequences in $\calR$. 
We define three functors $s^{\calR}$, $m^{\calR}$ and $q^{\calR}$ from 
$E(\calR)$ to $\calR$ which sends an exact sequence 
$\bE \to \bF \to \bG$ to $\bE$, $\bF$ and $\bG$ respectively
\end{df}

\begin{para}[\bf Quasi-weak equivalences]
\label{para:quasi-weak equiv}
Let 
$P^{\#}:\Ch^{\#}(\cE) \to \calD^{\#}(\bE)$ 
be the canonical quotient functor. 
We denote the pull-back of the class of all isomorphisms in $\calD^{\#}(\bE)$ 
by $qw^{\#}$ or simply $qw$. 
We call a morphism in $qw$ a {\it quasi-weak equivalence}. 
We write $\Ch^{\#}(\bE)$ for a pair $(\Ch^{\#}(\cE),qw)$. 
We can easily prove that 
$\Ch^{\#}(\bE)$ is a complicial biWaldhausen 
category in the sense of \cite[1.2.11]{TT90}. 
In particular, it is a relative exact category. 
The functor $P^{\#}$ induces an equivalence of triangulated categories 
$\cT(\Ch^{\#}(\cE),qw)\isoto\calD^{\#}(\bE)$ 
(See \cite[3.2.17]{Sch11}). 
If $w$ is the class of all isomorphisms in $\cE$, 
then $qw$ is just the class of all quasi-isomorphisms in $\Ch^{\#}(\cE)$ 
and we denote it by $\qis$. 
\end{para}

\begin{cor}
\label{cor:fully faithful}
Let $\bE=(\cE,w)$ be a very strict relative exact category and 
$\#\in \{b,\pm,\operatorname{nothing}\}$. 
Then\\
$\mathrm{(1)}$ 
The inclusion functor $\cE^w \rinc \cE$ induces 
a fully faithful embedding $\calD^{\#}\cE^w \rinc \calD^{\#}\cE$.\\
$\mathrm{(2)}$ 
The inclusion functor $\Ch^{\#}\cE^w\rinc\Ch^{\#}\cE$ 
and the identity functor on $\Ch^{\#}\cE$ 
induce an exact sequence of relative exact categories 
$(\Ch^{\#}\cE^w,\qis)\to (\Ch^{\#}\cE,\qis)\to \Ch^{\#}\bE$. 
\qed
\end{cor}

\begin{para}[\bf Consistent axiom]
\label{para:consis axiom}
Let $\bE=(\cE,w)$ be a strict relative exact category. 
There exists the canonical functor $\iota_{\cE}^{\#}:\cE \to \Ch^{\#}(\cE)$ 
where $\iota_{\cE}^{\#}(x)^k$ is $x$ if $k=0$ and $0$ if $k\neq 0$. 
We say that $w$ (or $\bE$) satisfies 
the {\it consistent axiom} 
if $\iota_{\cE}^b(w)\subset qw$. 
We denote the full subcategory of 
consistent relative exact categories in $\RelEx$ by 
$\RelEx_{\operatorname{consist}}$. 
\end{para}

\begin{ex}
\label{ex:semi devices}
(\cf \cite{Moc11b}). 
$\mathrm{(1)}$ 
A pair $(\cE,i_{\cE})$ of 
an exact category $\cE$ with the class of all isomorphisms $i_{\cE}$ 
is a very strict consistent relative exact category.\\
$\mathrm{(2)}$ 
In particular we denote the trivial exact category by $0$ and 
we also write $(0,i_0)$ for $0$. 
$0$ is the zero objects in the category of consistent relative exact categories.\\
$\mathrm{(3)}$ 
A complicial exact category with weak equivalences in the sense of 
\cite[3.2.9]{Sch11} is a consistent relative exact category. 
In particular for any relative exact category $\bE$, 
$\Ch^{\#}(\bE)$ is a very strict consistent relative exact category.
\end{ex}

\begin{thm}[\bf Derived Gillet-Waldhausen theorem]
\label{thm:der GW}
{\rm (\cf \cite[4.15]{Moc11b}).} 
Let $\bE$ be a consistent relative exact category. 
Then\\
$\mathrm{(1)}$ 
The canonical functor 
$\iota_{\Ch^{\#}(\cE)}:\Ch^{\#}(\cE) \to \Ch^b\Ch^{\#}(\cE)$ 
induces an equivalence of triangulated categories 
$\calD^{\#}(\bE)\isoto \cT(\Ch^{\#}(\cE),qw) \isoto \calD^b(\Ch^{\#}(\bE))$.\\
$\mathrm{(2)}$ 
In particular, the canonical functor $\iota^b_{\cE}:\cE \to \Ch^b(\cE)$ 
induces an equivalence 
of triangulated categories $\calD^b(\bE)\isoto \calD^b(\Ch^b(\bE))$. 
\qed
\end{thm}

\begin{df}[\bf Quotient of consistent relative exact categories]
For any fully faithful relative functor 
$f:\bE=(\cE,w) \to \bF=(\cF,v)$ between 
consistent relative exact categories 
such that induced functor $\calD^b(f)$ is also fully faithful, 
we define a {\it quotient $\bF/\bE:=(\Ch^b(\cF),v/w)$ of $\bF$ by $\bE$} 
({\it along $f$}) as follows. 
There exists a canonical quotient morphism 
$P_{\bF/\bE}:\Ch^b(\cF) \to \calD^b(\cF)/\calD^b(\cE)$. 
We write $v/w$ 
for the pull back of all isomorphisms in $\calD^b(\cF)/\calD^b(\cE)$ 
by $P_{\bF/\bE}$. 
We put $\bF/\bE:=(\Ch^b(\cF),v/w)$. 
$\bF/\bE$ is again a consistent relative exact category by 
\cite[\S 4]{Moc11b}. 
We have the canonical relative functor 
$\iota_F^b:\bF \to \bF/\bE$.
\end{df}

\section{Localizing theory}
\label{sec:Localizing theory}

In this section, 
we will prove Theorem~\ref{thm:main thm 2}.

\begin{para}
\label{para:Rel ex seq} 
Let $\bE=(\cE,w)$ be a relative exact category. 
We denote the exact category of admissible short exact sequences in $\cE$ 
by $E(\cE)$. 
There exist three exact functors $s$, $t$ and $q$ from 
$E(\cE) \to \cE$ which send 
an admissible exact sequence $x\rinf y \rdef z$ to 
$x$, $y$ and $z$ respectively. 
We write $w_{E(\bE)}$ for the class of morphisms 
$s^{-1}(w)\cap t^{-1}(w)\cap q^{-1}(w)$ and 
put $E(\bE):=(E(\cE),w_{E(\bE)})$. 
We can easily prove that $E(\bE)$ is a relative exact category 
and the functors $s$, $t$ and $q$ are relative exact functors 
from $E(\bE)$ to $\bE$. 
Moreover 
we can easily prove that 
if $\bE$ is consistent, then 
$E(\bE)$ is also consistent.
\end{para}

\sn
Now we give a definition of additive theories which is slightly 
different from \cite[6.9, 7.8]{Moc11b}.

\begin{df}[\bf Additive theory]
\label{df:additive theory}
$\mathrm{(1)}$ 
A full subcategory $\calR$ of $\RelEx$ 
is {\it closed under extensions} if 
$\calR$ contains the trivial relative exact category $0$ and 
if for any $\bE$ in $\calR$, $E(\bE)$ is also in $\calR$.\\
$\mathrm{(2)}$ 
Let $F$ 
be a functor from a full subcategory $\calR$ of $\RelEx$ 
closed under extensions to an additive category $\cB$. 
We say that $F$ is an {\it additive theory} if 
for any relative exact category $\bE$ in $\calR$, 
the following projection is an isomorphism 
$$
\displaystyle{\begin{pmatrix}F(s)\\ F(q)\end{pmatrix}:
F(E(\bE)) \to F(\bE)\oplus F(\bE)}.
$$
\end{df}

\begin{lem}[\bf Eilenberg Swindle]
\label{lem:ES}
Let $\calR$ be a full subcategory of $\RelEx_{\operatorname{strict}}$ 
closed under extensions, 
$F$ a categorical homotopy invariant additive theory 
from $\calR$ to an additive category $\cB$ 
and $\bE$ a strict relative exact category in $\calR$. 
We assume that $\Ch^{+}\bE$ {\rm (}resp. $\Ch^{-}\bE${\rm )} 
is also in $\calR$. 
Then $F(\Ch^{+}\bE)$ 
{\rm (}resp. $F(\Ch^{-}\bE)${\rm )} is trivial. 
\end{lem}

\begin{proof}[\bf Proof]
We only give a proof for $\Ch^{+}\bE$. 
We denote $f:\Ch^{+}\bE \to \Ch^{+}\bE$ to be a relative exact functor by 
sending an object $x$ to $\displaystyle{\bigoplus_{n\geq 0}x[2n] }$. 
Then we have the equality $F(f[2])+F(\id_{\Ch^{+}\cE})=F(f)$ and $F(f[2])=F(f)$ 
by \cite{Moc11b} Proposition~7.9. 
Hence we obtain the result. 
\end{proof}

\begin{df}[\bf Localization theory]
\label{df: Loc th}
A {\it localizing theory} $(F,\partial)$ 
from a full subcategory $\calR$ of $\RelEx_{\operatorname{strict}}$ to 
a triangulated category $(\cT,\Sigma)$ is a pair 
of functor $F:\calR \to \cT$ and a natural transformation 
$\partial:Fq \to \Sigma Fs$ between functors 
$E(\calR) \substack{\overset{s}{\to}\\ \underset{q}{\to}} \calR \onto{F} \cT$ 
which sends a exact sequnece $\bE \onto{i} \bF \onto{j}\bG$ in $\calR$ 
to a distingushed triangle 
$F(\bE)\onto{F(i)}F(\bF)\onto{F(j)}F(\bG)\onto{\partial_{(i,j)}}\Sigma F(\bE)$ in $\cT$.
\end{df}

\begin{rem}
\label{rem:localizing theory}
$\mathrm{(1)}$ 
The non-connective $K$-theory on $\RelEx_{\operatorname{consist}}$ 
studied in \cite{Moc11b} 
is a categorical homotopy invariant localization theory.\\
$\mathrm{(2)}$ 
(\cf \cite[7.9]{Moc11b}). 
Let $F$ be a localization theory on a full subcategory $\calR$. 
Then\\
$\mathrm{(i)}$ 
$F$ is a derived invariant functor. 
Namely 
if a morphism $\bE \to \bF$ in $\calR$ 
induces an equivalence of triangulated 
categories $\calD^b\bE \to \calD^b\bF$, 
then the induced morphism 
$F(\bE)\to F(\bF)$ is an isomorphism. 
In particular if $\iota_{\cE}^b:\bE \to \Ch^b\bE$ is in $\calR$, 
then $F(\iota_{\cE}^b)$ is an isomorphism.\\
$\mathrm{(ii)}$ 
If further we assume that $\calR$ is closed under extensions and 
if $F$ is categorical homotopy invariant, 
then we can easily prove that $F$ is an additive theory. 
\end{rem}

\begin{thm}
\label{thm:abstract main theorem}
Let $(F,\partial)$ be 
a categorical homotopy invariant localizing theory from a 
full 
subcategory 
$\calR$ closed under extensions 
to a triangulated category $(\cT,\Sigma)$, 
$\bE$ a very strict relative exact category in $\calR$. 
Assume that $\Ch^{\#}\bF$ is also in $\calR$ 
for any $\#\in\{b,\pm,\operatorname{nothing}\}$ and 
for any $\bF\in\{\bE, (\cE,i_{\cE}), (\cE^w,i_{\cE^w})\}$. 
Then there is an isomorphism $F\Ch\bE \to \Sigma F\Ch^b\bE$. 
In particular if 
further we assume that 
$\bE$ is consistent, 
then we have an isomorphism $F\Ch\bE\isoto\Sigma F\bE$.
\end{thm}

\begin{proof}[\bf Proof]
First assume that $w$ is the class of all isomorphisms in $\cE$. 
Then the fully faithfull embeddings 
$\calD^b(\cE)\rinc \calD^{\pm}(\cE)$ and $\calD^{\pm}(\cE)\rinc \calD(\cE)$ 
yield the commutative diagram of distinguished triangles 
$${\footnotesize{\xymatrix{
F(\Ch^b\cE,\qis) \ar[r] \ar[d] & 
F(\Ch^+\cE,\qis) \ar[r] \ar[d] & 
F((\Ch^+\cE,\qis)/(\Ch^b\cE,\qis)) \ar[r]^{\ \ \ \ \ \ \ \ \ \ \ \textbf{I}} 
\ar[d]_{\textbf{III}} & 
\Sigma F(\Ch^b\cE,\qis) \ar[d]\\
F(\Ch^-\cE,\qis) \ar[r] & 
F(\Ch\cE,\qis) 
\ar[r]_{\!\!\!\!\!\!\!\!\!\!\!\!\!\!\!\!\!\!\!\!\!\!\textbf{II}}  & 
F((\Ch\cE,\qis)/(\Ch^-\cE,\qis)) \ar[r] & 
\Sigma F(\Ch^-\cE,\qis).
}}}$$
Here the morphisms $\textbf{I}$ and $\textbf{II}$ are isomorphisms by 
triviality of $F(\Ch^{\pm}\cE,\qis)$ and the morphism $\textbf{III}$ is also an isomorphism 
by \cite{Moc11b} Proposition~7.10 $\mathrm{(2)}$. 
We denote the compositions of the morphisms $\textbf{I}$ and the inverse of $\textbf{III}$ 
and $\textbf{II}$ by $\Delta_{\cE}:F(\Ch\cE,\qis) \to \Sigma F(\Ch^b\cE,\qis)$. 
Then $\Delta_{\cE}$ is functorial on $\cE$. 

\sn
Next we consider the general case. 
By virtue of Corollary~\ref{cor:fully faithful} $\mathrm{(2)}$, 
there is a commutative diagram of distingushed triangles 
$${\footnotesize{\xymatrix{
F(\Ch\cE^w,\qis) \ar[r] \ar[d]^{\wr}_{\Delta_{\cE^w}} & 
F(\Ch\cE,\qis) \ar[r] \ar[d]_{\wr}^{\Delta_{\cE}} & 
F(\Ch\bE) \ar[r] \ar@{-->}[d]^{\textbf{IV}} & 
\Sigma F(\Ch\cE^w,\qis) \ar[d]_{\wr}^{\Sigma \Delta_{\cE^w}}\\
\Sigma F(\Ch^b\cE^w,\qis) \ar[r] & 
\Sigma F(\Ch^b\cE,\qis) \ar[r] & 
\Sigma F(\Ch^b\bE) \ar[r] & 
\Sigma^2 F(\Ch^b\cE^w,\qis).
}}}$$
Then there exists the dotted morphism $\textbf{IV}$ in the diagram above which 
makes the diagram commutative and it is an isomorphism by $5$-lemma. 
\end{proof}

\begin{rem}
\label{rem:main thm}
The full subcategory 
$\RelEx_{\operatorname{consist}}$ satisfies the assumption in 
Theorem~\ref{thm:abstract main theorem} by \cite{Moc11b}. 
In particular we obtain Theorem~\ref{thm:main thm 2}.
\end{rem}

\section{Higher derived categories}
\label{sec:Main theorems}

\sn
In this section, we assume that $\bE$ is a very strict consistent relative exact category.

\begin{para}[\bf Higher derived categories]
\label{para:higehr derived categories}
Let us denote $n$-th times iteration of $\Ch$ for $\bE$ 
by $\Sigma^{n}\bE$ and 
$\calD_n(\bE):=\calD^b(\Sigma^{n}\bE)$ 
the {\it{$n$-th higher derived category}} of $\bE$. 
Then\\
$\mathrm{(1)}$ 
$\calD_0(\bE)$ is 
just the usual bounded derived category $\calD^b(\bE)$ of $\bE$.\\
$\mathrm{(2)}$ 
For any positive integer $n$, 
$\calD_n(\bE)$ is 
the unbounded derived category
$\calD(\Sigma^{n-1}\bE)$ of $\Sigma^{n-1}\bE$ by 
Theorem~\ref{thm:der GW} $\mathrm{(1)}$. 
In particular, $\calD_{1}(\bE)$ 
is just the unbounded derived category $\calD(\bE)$ of $\bE$. 
\end{para}

\sn
As the following corollary, 
we can consider the negative $K$-groups as an obstruction group of 
idempotent completeness of the higher derived categories:

\begin{cor}
\label{cor:obst}
For any positive integer $n$, we have\\
$\mathrm{(1)}$ 
$\bbK_{-n}(\bE) \simeq \bbK_0(\calD_n(\bE))$.\\
$\mathrm{(2)}$ 
$\bbK_{-n}(\bE)$ is trivial if and only if 
$\calD_n(\bE)$ is idempotent complete.
\end{cor}

\begin{proof}[\bf Proof]
By Theorem \ref{thm:main thm 2}, 
we have 
$\bbK_{-n}(\bE)\simeq \bbK_0(\Ch_n(\bE))\simeq \bbK_0(\calD_n(\bE))$.
Then Proposition \ref{cor:vanish} below leads the desired assertion.  
\end{proof}

\begin{prop}
\label{cor:vanish} 
$\mathrm{(1)}$ 
For an essentially small triangulated category $\cT$, 
if $\bbK_0(\cT)=K_0(\cT^{\sim})$ is trivial, 
then $\cT$ is idempotent complete.\\
$\mathrm{(2)}$ 
  The derived category $\calD(\bE)$ is idempotent complete if and only if 
  the Grothendieck group $\bbK_0(\calD(\bE))=K_0(\calD(\bE)^{\sim})$ is trivial. 
\end{prop}

\begin{proof}[\bf Proof]
$\mathrm{(1)}$ 
Since the map $K_0(\cT) \to K_0(\cT^{\sim})$ is injective 
by \cite{Tho97} Corollary\ 2.3, 
now $K_0(\cT)$ is also trivial. 
Applying the Thomason classification theorem of 
  (strictly) dense triangulated subcategories in 
  essentially small triangulated categories \cite{Tho97} Theorem\ 2.1 
  for $\cT^{\sim}$, 
  the inclusion functor $\cT \to \cT^{\sim}$ 
  must be an equivalence.  

\sn
$\mathrm{(2)}$ 
We have the equalities 
$K_0(\Ch(\cE^w),\qis)=K_0(\Ch(\cE),\qis)=0$ 
as in the proof of Collorary~6 in \cite{Sch06}. 
Therefore $K_0(\calD(\bE))=K_0(\Ch(\bE))=0$ 
by the canonical fibration sequence associated to 
the exact sequence in 
Corollary~\ref{cor:fully faithful} $\mathrm{(2)}$ 
for $\Ch\bE$. 
If $\calD(\bE)$ is idempotent complete, that is, 
$\calD(\bE)\isoto\calD(\bE)^{\sim}$, then we have 
$\bbK_0(\calD(\bE))=K_0(\calD(\bE)^{\sim})=K_0(\calD(\bE))=0$. 
The converse is followed from $\mathrm{(1)}$.
\end{proof}

\end{document}